\documentclass[12pt,reqno]{amsart}
\usepackage{times}
\usepackage[T1]{fontenc}
\usepackage{mathrsfs}
\usepackage{latexsym}
\usepackage[dvips]{graphics}
\usepackage[titletoc,title]{appendix}
\usepackage{epsfig}
\usepackage{amsmath,amsfonts,amsthm,amssymb,amscd}
\input amssym.def
\input amssym.tex
\usepackage{color}
\usepackage{hyperref}
\usepackage{color}
\usepackage{breakurl}
\usepackage{comment}
\newcommand{\bburl}[1]{\textcolor{blue}{\url{#1}}}
\newcommand{\be}{\begin{equation}}
\newcommand{\ee}{\end{equation}}
\newcommand{\bea}{\begin{eqnarray}}
\newcommand{\eea}{\end{eqnarray}}
\newtheorem{thm}{Theorem}[section]

\newtheorem{cor}[thm]{Corollary}
\newtheorem{lem}[thm]{Lemma}
\newtheorem{prop}[thm]{Proposition}
\newtheorem{exa}[thm]{Example}
\newtheorem{defi}[thm]{Definition}

\newtheorem{rek}[thm]{Remark}
\newtheorem{que}[thm]{Question}

\numberwithin{equation}{section}

\providecommand{\floor}[1]{\lfloor #1 \rfloor}

\title{On Sets with More Products than Quotients}

\author{H\`ung Vi\d{\^e}t Chu}
\address{Department of Mathematics, University of Illinois at Urbana-Champaign, Urbana, Illinois 61820, USA.}
\email{\textcolor{blue}{\href{mailto:chuh19@mail.wlu.edu}{chuh19@mail.wlu.edu}}}

\date{\today}

\begin{document}

\begin{abstract}
Given a finite set $A\subset \mathbb{R}\backslash \{0\}$, define 
\begin{align*}&A\cdot A \ =\ \{a_i\cdot a_j\,|\, a_i,a_j\in A\},\\
&A/A  \ =\ \{a_i/a_j\,|\,a_i,a_j\in A\},\\
&A + A \ =\ \{a_i + a_j\,|\, a_i,a_j\in A\},\\
&A - A  \ =\ \{a_i - a_j\,|\,a_i,a_j\in A\}.\end{align*}
The set $A$ is said to be MPTQ (more product than quotient) if $|A\cdot A|>|A/A|$ and MSTD (more sum than difference) if $|A + A|>|A - A|$. Since multiplication and addition are commutative while division and subtraction are not, it is natural to think that MPTQ and MSTD sets are very rare. However, they do exist. This paper first shows an efficient search for MPTQ subsets of $\{1,2,\ldots,n\}$ and proves that as $n\rightarrow \infty$, the proportion of MPTQ subsets approaches $0$. Next, we prove that MPTQ sets of positive numbers must have at least $8$ elements, while MPTQ sets of both negative and positive numbers must have at least $5$ elements. Finally, we investigate several sequences that do not have MPTQ subsets. 
\end{abstract}

\maketitle

\section{Introduction and Main Results}
\subsection{Introduction}
Given a finite set $A\subset \mathbb{R}\backslash \{0\}$, define 
\begin{align*}&A\cdot A \ =\ \{a_i\cdot a_j\,|\, a_i,a_j\in A\},\\
&A/A  \ =\ \{a_i/a_j\,|\,a_i,a_j\in A\}.\end{align*}
The set $A$ is said to be MPTQ (more product than quotient) if $|A\cdot A|>|A/A|$, quotient-dominated if $|A\cdot A|<|A/A|$, and balanced if $|A\cdot A| = |A/A|$. Also, define \begin{align*}A + A \ =\ \{a_i + a_j\,|\, a_i,a_j\in A\},\\A - A  \ =\ \{a_i - a_j\,|\,a_i,a_j\in A\}.\end{align*}
The set $A$ is said to be MSTD (more sum than difference) if $|A + A|>|A - A|$.
We consider MPTQ and MSTD subsets of $\mathbb{R}$ (instead of $\mathbb{N}$ as in previous work) because this extension allows us to define the log transformation and the exponential transformation, which are crucial in describing the relationship between the two types of sets. Since multiplication and addition are commutative while division and subtraction are not, it is natural to think that MPTQ and MSTD sets are very rare. However, they do exist. It is believed that Conway (1969) was the first to give an example of the MSTD set
$$\{0,2,3,4,7,11,12,14\},$$
whose sum set and difference set have $26$ elements and $25$ elements, respectively. The number of MSTD subsets of $\{1,2,\ldots,n\}$ grows quite quickly as $n$ grows. On the other hand, it is harder to find MPTQ subsets of $\{1,2,\ldots, n\}$ because $\{1,2,\ldots,36\}$ contains no MPTQ subsets. Hence, we instead look for MPTQ subsets of $\{2^n\cdot 3^m\,|\,0\le n,m\le 6\}$. Below are several sets we found
\begin{align*}
    &\{12,27,36,96,108,144,162,243,648,864,1944\},\\
    &\{8,18,32,36,48,216,324,432,486,864,1944\},\\
    &\{4,9,12,32,36,48,54,81,216,288,648\},\\
    &\{1,6,8,9,24,72,108,288,324,432,2592\},\\
    &\{3,18,24,27,72,108,324,864,972,1296,7776\}.
\end{align*}
Surprisingly, in 2007, Martin and O'Bryant \cite{MO} proved that as $n\rightarrow \infty$, the proportion of MSTD subsets of $\{1,2,\ldots,n\}$ is bounded below by a positive constant. Since then, research on sum-dominant sets has made considerable progress: see \cite{FP, Ma, Na2, Ru1, Ru2, Ru3} for
history and overview, \cite{He,MOS,MS,Na3,Zh1} for explicit constructions, \cite{CLMS2, MO, Zh3} for positive lower bound for the percentage of
sum-dominant sets, and \cite{CI1,CLMS1,MV,Zh2} for extensions to other settings. However, much less is known about MPTQ sets. Fortunately, many results on MSTD sets hold for MPTQ sets with the use of the log transformation and exponentiation of sets. The goal of this paper is to provide an understanding of MPTQ sets through both what we know about MSTD sets and unique properties of MPTQ sets themselves. Furthermore, properties of MPTQ sets also shed light on new results about MSTD sets. We focus on the four topics: \textit{how to search for MPTQ subsets of $\{1,2,\ldots,n\}$ more efficiently}, \textit{the probability measure of MPTQ subsets of $\{1,2,\ldots,n\}$}, \textit{when sets are not MPTQ}, and \textit{what sequences do not contain MPTQ subsets}. 

\subsection{Notation}
We first introduce some notation. 
\begin{enumerate}
    \item For $n\in\mathbb{N}$ and $r\in \mathbb{R}\backslash\{0,\pm 1\}$, define $G_{n,r} = \{1,r,r^2,\ldots,r^{n-1}\}$.
    \item For $(a_i)_{i=1}^\ell$ and a set $A$, we write $(a_i)_{i=1}^\ell\rightarrow A$ to mean the introduction of $\ell$ numbers $(a_i)_{i=1}^\ell$ into the set $A$ to form $A\cup\{a_i \mid  1\le i\le \ell\}$. (We assume that $a_i\notin A$ for all $1\le i\le \ell$.)
    \item Given a set $A$ of positive real numbers and $1\neq r>0$, define 
    $$\log_r A \ =\ \{\log_r a_i\,|\,a_i\in A\}.$$
    Because $A$ contains only positive numbers, $\log_r A$ is well-defined and $|\log_r A| = |A|$. We call this the \textit{$r$-log transformation} of $A$.
   \item  Given a set $B$ of real numbers and $1\neq r > 0$, define
    $$r^B \ =\ \{r^{b_i}\,|\, b_i\in B\}.$$
    Because $1\neq r>0$, $|r^B| = |B|$. We call this the \textit{$r$-exponential transformation} of $B$.
    \item Let $A = \{a_1,a_2,\ldots,a_n\}$, where $|a_1| \le |a_2| \le \cdots \le |a_n|$. We write $A$ in the following form 
$$A \ =\ (a_1\,|\,a_2/a_1, a_3/a_2,\ldots, a_n/a_{n-1}).$$ All information about set $A$ is preserved in this notation. Call $$a_2/a_1,a_3/a_2,\ldots,a_n/a_{n-1}$$ \textit{a multiplier sequence}. Note that the absolute value of each quotient in a multiplier sequence is at least $1$ and one set may have more than one multiplier sequence as shown in Example \ref{>1sq}.

\begin{exa} \label{>1sq} \normalfont
Let $A = \{5,1280,-10,-40,40,2560,160,320\}$. We can write $$A \ =\ (5\,|\,-2,4,-1,4,2,4,2)$$ or
$$A \ =\ (5\,|\,-2,-4,-1,-4,2,4,2).$$
\end{exa}
\end{enumerate}
\subsection{Main results}
\begin{thm}\label{esearch}
Let $n\in \mathbb{N}_{\ge 4}$. If we want to find all MPTQ subsets of $\{1,2,\ldots,2n\}$, we need to check at most $2^{2n-t}$ subsets, where $t$ is the number of primes strictly between $n$ and $2n$. 
\end{thm}
With a simple program, the author found no MPTQ subsets of $\{1,2,\ldots,36\}$, and the program reported a memory error when we attempted $\{1,2,\ldots,38\}$. Recall that $\{1,2,\ldots, 15\}$ already contains several MSTD sets, so MPTQ subsets appear much later than MSTD sets. Along with \cite[Theorem 1]{MO}, the following theorem shows that MPTQ sets are rare compared to MSTD sets. 
\begin{thm}\label{goto0}
As $n\rightarrow \infty$, the proportion of MPTQ subsets of $\{1,2,\ldots,n\}$ approaches $0$; that is, as $n\rightarrow \infty$, almost all subsets of $\{1,2,\ldots,n\}$ are not MPTQ. 
\end{thm}

Our next result concerns the smallest cardinality of MPTQ sets, comparably to \cite[Theorem 1]{He} by Hegarty.
\begin{thm}\label{smallestcar}
Let $A$ be a MPTQ set of real numbers. The following claims are true.
\begin{enumerate}
    \item If $A$ contains only positive numbers, then $|A|\ge 8$. 
    \item If $A$ contains negative numbers, then $|A|\ge 5$.
\end{enumerate}  
\end{thm}
When we allow negative numbers to be included, the proof for the smallest cardinality becomes more complicated very quickly. 
\begin{que}\label{quesmallest}\normalfont
What is the smallest cardinality among MPTQ sets of real numbers? 
\end{que}
To prove \cite[Theorem 1]{He}, Hegarty (2007) used a nontrivial algorithm to reduce the problem to finite computation. The mathematica program was reported to run for about 15 hours. However, because it takes less memory and computation power for computers to do addition and subtraction than multiplication and division, Question \ref{quesmallest} is thus quite challenging. 

Lastly, we find sequences that do not contain MPTQ subsets. 
\begin{thm}\label{logprimenoMSTD} Let $P$ be the set of all primes. The following are true.  
\begin{enumerate}
\item The set $P$ contains no MPTQ subsets. 
\item Fix $1\neq r > 0$. Consider $P_r = \log_r(P)$. Then $P_r$ contains no MSTD subsets. 
\end{enumerate}
\end{thm}

\begin{thm}\label{thm:gen} Let $A = \{a_k\}_{k=1}^\infty$ be an increasing sequence in absolute value of real numbers. If there exists a positive integer $r$ such that
\begin{enumerate}
	\item $|a_k| > |a_{k-1}\cdot a_{k-r}|$ for all $k\ge r+1$, and
	\item $A$ does not contain any MPTQ set $S$ with $|S| \le 2r-1$,
\end{enumerate}
then $A$ contains no MPTQ set.
\end{thm}

Theorem \ref{thm:gen} is comparable to \cite[Theorem 1]{CI1} by H. V. Chu et al. but allows more flexibility in the sense that our sequence needs only to be increasing in absolute value. 

\begin{exa}\label{mulFi}\normalfont
Define the Fibonacci sequence to be $F_1 = 1$, $F_2 = 2$, and $F_{n} = F_{n-1} + F_{n-2}$ for $n\ge 3$. 
Let $A = \{a_k\}_{k=1}^\infty$ with $a_k = 2^{F_k}$. Because for $k\ge 4$, $a_k = a_{k-1}a_{k-2} > a_{k-1}a_{k-3}$, and there are no MPTQ sets of size $5$ due to Theorem \ref{smallestcar} item 1, $A$ has no MPTQ subsets. 
\end{exa}

\begin{exa}\label{mulFine}\normalfont
Let $A = \{a_k\}_{k=1}^\infty$ with $a_k = \pm k^{F_k}$ (we may choose the sign for each $a_k$ arbitrarily). Because for $k\ge 3$, $$|a_k|\ =\ k^{F_k} \ =\ k^{F_{k-1}}\cdot k^{F_{k-2}}\ >\ (k-1)^{F_{k-1}}\cdot (k-2)^{F_{k-2}} \ =\ |a_{k-1}a_{k-2}|,$$ and there are no MPTQ sets of size $3$ due to Theorem \ref{smallestcar} item 2, $A$ has no MPTQ subsets. 
\end{exa}

\begin{rek}\normalfont
It is interesting to see that while the set of prime numbers contains infinitely many MSTD subsets \cite[Theorem 5]{CI1}, it contains no MPTQ subsets. On the other hand, an example of a set containing infinitely many MPTQ subsets while no MSTD subsets is $\{1,2,2^2,2^3,\ldots\}$.\footnote{The reason that $\{1,2,2^2,2^3,\ldots\}$ has no MSTD subsets is due to \cite[Corollary 8]{CI1}.} Finally, we also have sets that contain neither MSTD nor MPTQ subsets. An example is the sequence in Example \ref{mulFi}.
\end{rek}

\section{Search for MPTQ subsets more efficiently and\\ Probability measure for MPTQ subsets}
\begin{defi}\normalfont
For every MPTQ set $A$, let $k$ be the largest positive integer (if any) such that 
$$|A\cdot A|-|A/A| \ge k|A|+\frac{k(k-3)}{2}+1.$$
Then $A$ is said to be $k$-special MPTQ.
\end{defi}
\begin{proof}[Proof of Theorem \ref{esearch}]
Fix $n\ge 4$. Let $t$ be the number of primes strictly between $n$ and $2n$. By Bertrand's postulate, we know that $t\ge 1$. Let $p$ be such a prime and $A$ be a subset of $\{1,2,\ldots,2n\}$ not containing $p$. We claim that $p\rightarrow A$ gives $|A|+1$ new products and $2|A|$ new quotients. We proceed by proving the claim. 

Write $A = \{a_1,a_2,\ldots, a_j\}$, where $a_1<a_2<\cdots<a_j$. Consider the following products
$$pa_1, pa_2, \ldots, pa_j, p^2.$$
They are all new products from $p\rightarrow A$. Indeed, suppose that there exists $1\le k,\ell,m\le j$ such that either 
$a_ka_\ell = pa_m$ or $a_ka_\ell = p^2$. In both cases, either $p|a_k$ or $p|a_\ell$, which contradicts that $n<p<2n$. So, the number of new products is exactly $j+1 = |A| + 1$. 
Consider the following quotients 
$$\frac{p}{a_1}, \frac{p}{a_2},\ldots, \frac{p}{a_j}.$$
They are all new quotients from $p\rightarrow A$. Indeed, suppose that there exists $1\le k,\ell, m\le j$ such that $\frac{p}{a_k} = \frac{a_\ell}{a_m}$. Then $pa_m = a_ka_\ell$ and so, either $p|a_k$ or $p|a_\ell$. Hence, $$\max \{a_\ell, a_k\}\ \ge\ 2p \ >\ 2n,$$ 
which contradicts that $A\subseteq \{1,2,\ldots, 2n\}$. Therefore, all the above quotients and their reciprocals are new. So, the number of new quotients is exactly $2j = 2|A|$. 

We have proved that $p\rightarrow A$ gives $|A|+1$ new products and $2|A|$ new quotients. For any $|A|\ge 1$, $2|A| \ge |A|+1$. So, given a MPTQ set containing some primes strictly between $n$ and $2n$, we know that by excluding these primes from the set, we still have a MPTQ set. 

Let $S$ be a MPTQ subset of $\{1,2,\ldots, n\}$ and let $k$ be the maximum number of primes strictly between $n$ and $2n$ that can be added to $S$ and we still have a MPTQ set $S'$. Applying our above claim repeatedly, we have
\begin{align*}|S'\cdot S'| &\ =\ |S\cdot S|+\sum_{i=1}^{k} (|S|+i)\\
|S'/S'|&\ =\ |S/S| + \sum_{i=0}^{k-1}2(|S|+i).\end{align*}
Because $|S'\cdot S'| - |S'/S'|\ge 1$, we have 
$$|S\cdot S| - |S/S|\ \ge\ k|S| + \frac{k(k-3)}{2} + 1.$$
So, $S$ is $k$-special MPTQ. 

We now outline the steps in finding all MPTQ subsets of $\{1,2,\ldots, n\}$. 
\begin{itemize}
    \item[(1)] Search for all MPTQ subsets of $\{1,2,\ldots, n\}$ without primes strictly between $n$ and $2n$. 
    \item[(2)] For each MPTQ subset $S$, find the largest positive integer $k$ such that $$|S\cdot S| - |S/S|\ \ge\ k|S| + \frac{k(k-3)}{2} + 1;$$
    in other words, classify all MPTQ subsets found in step (1) by their $k$-special MPTQ property. This can be done since from step (1), we already know $|S\cdot S|$, $|S/S|$ and $|S|$ of each MPTQ set $S$.  
    \item [(3)] Given a $k$-special MPTQ set, we can add at most $k$ primes strictly between $n$ and $2n$ to it and still have a MPTQ set. 
\end{itemize}
Following these steps, we will have all MPTQ subsets of $\{1,2,\ldots, n\}$. Therefore, the number of subsets we need to check is reduced by a factor of $2$ for each $p$. Because there are $t$ primes strictly between $n$ and $2n$, this
method helps reduce the number of subsets to be checked by a factor of $2^t$. \footnote{There are many improved versions of Bertrand's postulate, which may reduce the number of subsets to be checked further as our $n$ grows. For example, Nagura \cite{Nagura} proved that for $n\ge 25$, there is always a prime between $n$ and $6n/5$. Therefore, between $n$ and $2n$, there are at least $2$ primes. This reduces the number of subsets to be checked by a factor of $4$. } 
\end{proof}
\begin{exa}\normalfont
If we want to find all MPTQ subsets of $\{1,2,\ldots,36\}$, we can instead find all MPTQ subsets of $\{1,2,\ldots,36\}\backslash \{19,23,29,31\}$.
\end{exa}

\begin{proof}[Proof of Theorem \ref{goto0}]
Due to \cite[Corollary 1.1]{sa} by Cilleruelo and Guijarro-Ordonez, almost all sets $A\subseteq \{1,2,\ldots,n\}$ have $|A/A| \sim Cn^2$ for some constant $C>0$. On the other hand, Erd\H{o}s \cite{Er} proved that as $n\rightarrow \infty$, $$|A\cdot A| \ \le\ |\{1,2,\ldots,n\}\cdot \{1,2,\ldots, n\}| \ =\ \frac{n^2}{(\log n)^{\delta+o(1)}} \ =\ o(n^2),$$
where $\delta = 1- \frac{1+\log\log 2}{\log 2}$. Therefore, as $n\rightarrow \infty$, almost all subsets of $\{1,2,\ldots,n\}$ are not MPTQ. 
\end{proof}

\section{Preliminaries}
We now mention some important properties of MPTQ sets and the relationship between MSTD and MPTQ sets. 
\begin{defi}\normalfont
A set $A$ is symmetric with respect to $a$ if there exists $a\in \mathbb{R}\backslash \{0\}$ such that $a/A = A$. 
\end{defi}
\begin{exa}\normalfont
The set $S_1 = \{3,4,6,8,9,27,48,144,162,216,324,432\}$ is symmetric with respect to $1296$ because 
$$S_1 \ =\ \bigg\{\frac{1296}{3},\frac{1296}{4},\frac{1296}{6},\frac{1296}{8},\frac{1296}{9},\frac{1296}{27},\frac{1296}{48},\frac{1296}{144},\frac{1296}{162},\frac{1296}{216},\frac{1296}{324},\frac{1296}{432}\bigg\}.$$
\end{exa}
\begin{lem}
A symmetric set is balanced. 
\end{lem}
\begin{proof}
Let $A$ be a symmetric set with respect to $a$. We have $$|A\cdot A| \ =\ |(a/A)\cdot A| \ =\ |a\cdot (A/A)| \ =\ |A/A|.$$
Therefore, $A$ is balanced. 
\end{proof}
\begin{rek}\normalfont
Let $A = \{a_1,\ldots,a_n\}$ be a MPTQ set and $A^p$ be the nonempty subset of $A$ whose elements are divisible by a prime $p$. Let $q$ be a prime that does not divide any number in $A$. For each number in $A^p$, if we replace $p$ in its prime factorization by $q$ to form $(A^p)'$. Then $(A\backslash A^p)\cup (A^p)'$ is MPTQ. The reason is that the process does not change the sizes of the product set and the quotient set. MSTD sets do not enjoy this property. We call this the \textit{$(p,q)$-prime switch} of $A$.
\end{rek}
\begin{exa}\normalfont
The set
$$S_2 \ =\ \{3,4,6,8,9,27,48,72,144,162,216,324,432\}$$
is MPTQ. By the $(2,5)$-prime switch, we have the new set
$$S_3 \ =\ \{3,25,15,125,9,27,1875,1125,5625,405,3375,2025,16875\},$$
which is also MPTQ.
\end{exa}
\begin{defi}\normalfont
Let $A\in \mathbb{R}\backslash \{0\}$. For $a_i,a_j\in A$, we have $a_i/a_i = a_j/a_j = 1$. We call the pair $(a_i, a_i)$, $(a_j,a_j)$ a trivial pair of equal quotients. 
\end{defi}
\begin{prop}\label{trivialbounds}
For a finite set $A\in\mathbb{R}\backslash\{0\}$, we have the following trivial bounds
\begin{align}
    |A\cdot A|&\ \le\ \frac{|A|(|A|+1)}{2},\label{eq1}\\
    |A/A| &\ \le\ |A|(|A|-1)+1\label{eq2}.
\end{align}
The equality in (\ref{eq1}) is achieved if every pair of numbers gives a distinct product, and the equality in (\ref{eq2}) if every pair of distinct numbers gives a distinct quotient. 
\end{prop}

\begin{rek} \label{countquot}\normalfont
Given a set $A\in\mathbb{R}\backslash\{0\}$, for each $q\in A/A$, define 
$$(A/A)_q \ =\ \{\{a_i,a_j\}\,|\, a_i/a_j = q\mbox{ and } a_i,a_j\in A\}.$$
Then 
\begin{align}\label{numberof=quo}\frac{1}{2}(|A|(|A|-1)+1-|A/A|) \ =\ \sum_{q\in A/A, q\neq 1, |q|\ge 1}(|(A/A)_q| - 1).\end{align}
The part $|A|(|A|-1)+1$ comes from Inequality (\ref{eq2}).
\end{rek}
We provide an example to help understand (\ref{numberof=quo}).
\begin{exa}\normalfont
Let $A = \{1,2,3,6,9\}$. We have 
$$A/A = \bigg\{1,2,3,\frac{3}{2},\frac{9}{2},6,9,\frac{1}{2},\frac{1}{3},\frac{1}{9},\frac{1}{6},\frac{2}{3},\frac{2}{9}\bigg\}$$
and so, $|A/A|=13$. The left side of (\ref{numberof=quo}) is $4$. Consider the right side of (\ref{numberof=quo}). We have
\begin{align*}
    &(A/A)_2 \ =\ \{\{2,1\}, \{6,3\}\},\\
    &(A/A)_3 \ =\ \{\{3,1\}, \{6,2\}, \{9,3\}\},(A/A)_{3/2} \ =\ \{\{3,2\}, \{9,6\}\},\\
    &(A/A)_{9/2}\ =\ \{\{9,2\}\}, (A/A)_{6} \ =\{\{6,1\}\}, (A/A)_9 \ =\ \{\{9,1\}\}.
\end{align*}
The right side is $\sum_{q\in A/A, q > 1} (|(A/A)_q|-1) = 4$, as desired. 
\end{exa}
\begin{rek} \label{countprod}\normalfont
Given a set $A\in \mathbb{R}\backslash \{0\}$, for each $p\in A\cdot A$, define
$$(A\cdot A)_p \ =\ \{\{a_i,a_j\}\,|\, a_ia_j = p\mbox{ and } a_i,a_j \in A\}.$$
Then
\begin{align}\frac{1}{2}|A|(|A|+1) - |A\cdot A| \ =\ \sum_{p\in A\cdot A}(|(A\cdot A)_p|-1).\label{numberof=pro}\end{align}
The part $\frac{1}{2}|A|(|A|+1)$ comes from Inequality (\ref{eq1}).
\end{rek}
\begin{exa}\normalfont
Let $A = \{1,2,3,6,9\}$. We have
\begin{align*}
    A\cdot A \ =\ \{1,2,3,4,6,9,12,18,27,36,54,81\}
\end{align*}
and so, $|A\cdot A| = 12$. The left side of (\ref{numberof=pro}) is $3$. Consider the right side of (\ref{numberof=pro}). We have
\begin{align*}
    &(A/A)_1 = \{\{1,1\}\},  (A/A)_2 = \{\{1,2\}\},  (A/A)_3 = \{\{3\}\},  (A/A)_4 = \{\{2,2\}\},\\
    &(A/A)_6 = \{\{1,6\}, \{2,3\}\}, (A/A)_9 = \{\{1,9\},\{3,3\}\}, (A/A)_{12} = \{\{2,6\}\},\\
    &(A/A)_{18} = \{\{2,9\},\{3,6\}\},  (A/A)_{27} = \{\{3,9\}\},(A/A)_{36} = \{\{6,6\}\},\\ 
    &(A/A)_{54} = \{\{6,9\}\},  (A/A)_{81} = \{\{9,9\}\}.
\end{align*}
So, the right side is $3$, as desired. \end{exa}
\begin{rek}\label{deepex}\normalfont
Let $A\subset \mathbb{R}\backslash \{0\}$. Loosely speaking, Remark \ref{countquot} and Remark \ref{countprod} show how pairs of equal products and nontrivial pairs of equal quotients reduce $|A\cdot A|$ and $|A/A|$, respectively.  When we look at the reduction, we have to be very careful. 
For example, if we have $a_i \cdot  a_j = a_m \cdot a_n = a_p\cdot a_q$ for some $a_i, a_j, a_m, a_n, a_p, a_q\in A$ and $a_i,a_j, a_m, a_p,a_q$ being pairwise different, $|A\cdot A|$ is reduced by $2$, not $3$ even though $\{a_i,a_j\}, \{a_m,a_n\},\{a_p,a_q\}\in (A\cdot A)_{a_ia_j}$. This is why we need to subtract $1$ from each summand in (\ref{numberof=pro}). The same reasoning applies for $A/A$. Now, we investigate the relationship between the number of nontrivial pairs of equal quotients and the number of pairs of equal products. Consider two cases. 
\begin{enumerate}
\item Case 1: we do not have $a_i \cdot  a_j = a_m \cdot a_n = a_p\cdot a_q$ for all $a_i, a_j, a_m, a_n, a_p, a_q\in A$ and $a_i, a_j, a_m, a_p, a_q$ being pairwise different. In other words, for all $p\in A\cdot A$, $1\le |(A\cdot A)_p|\le 2$. In this case, we have a very useful inequality. Let $a_i,a_j,a_m,a_n\in A$, where $a_j/a_i = a_n/a_m\neq 1$ and $|a_i|\le |a_j|\le |a_m|\le |a_n|$.
\begin{itemize}
\item If $a_j\neq a_m$, we have another nontrivial pair of equal quotients whose absolute values are at least $1$: $a_m/a_i = a_n/a_j$. 
\item If $a_j = a_m$, then we do not have another pair. 
\end{itemize}
In both cases, we have $a_j\cdot a_m = a_i\cdot a_n$, a pair of equal products. So, a nontrivial pair of equal quotients whose absolute values are at least $1$ increases the right side of (\ref{numberof=quo}) by at most $2$, while its corresponding pair of equal products increases the right side of (\ref{numberof=pro}) by exactly $1$. Hence, if $$k = \sum_{q\in A/A, q\neq 1, |q|\ge 1}(|(A/A)_q| - 1),$$ then  
\begin{align}\label{www}\sum_{p\in A\cdot A}(|(A\cdot A)_p|-1)\ \ge\ k/2.\end{align}

\item Case 2: $a_i \cdot  a_j = a_m \cdot a_n = a_p\cdot a_q$ for some $a_i, a_j, a_m, a_n, a_p, a_q\in A$ and $a_i, a_j, a_m, a_p, a_q$ being pairwise different. Then we do not have (\ref{www}) anymore. To see why, suppose that  $\{1,4,5,8,10,40\}\subseteq A$. Then the following pairs of equal quotients 
\begin{align*}
 \frac{4}{1} \ =\ \frac{40}{10}, \frac{10}{1} \ =\ \frac{40}{4}, \frac{40}{8} \ =\ \frac{5}{1},\frac{40}{5} \ =\ \frac{8}{1}, \frac{5}{4} \ =\ \frac{10}{8},\frac{10}{5} \ =\ \frac{8}{4}.
\end{align*}
increase the right side of (\ref{numberof=quo}) by $6$.
The corresponding products given by these three pairs are 
$$4\cdot 10 \ =\  1\cdot 40,  1\cdot 40 \ =\ 5\cdot 8, 4\cdot 10 = 5\cdot 8.$$
As mentioned above, the right side of (\ref{numberof=pro}) only accounts for $2$ (not $3$) out of these three pairs of equal products since $4\cdot 10 = 1\cdot 40 = 5\cdot 8$. Because $6/2 = 3>2$, we do not have Inequality (\ref{www}).
\end{enumerate}
\end{rek}
\begin{lem}\label{expo}
Let a MSTD set $A$ be chosen. Then for all $1\neq r  > 0$, $B = r^A$ is MPTQ. 
\end{lem}
\begin{proof}
We will prove that $|B/ B| = |A-A|$ and $|B\cdot B| = |A+A|$. Given a difference $a_i-a_j$ for some $a_i, a_j\in A$, we have the corresponding quotient $r^{a_i}/r^{a_j}$. Let $a'_i, a'_j\in A$. Because $r\notin \{0,\pm 1\}$, $a_i - a_j = a'_i - a'_j$ if and only if $r^{a_i-a_j} = r^{a'_i-a'_j}$. Therefore, $|B/B| = |A-A|$. Similarly, given a sum $a_p + a_q$ for some $a_p, a_q\in A$, we have the corresponding product $r^{a_p}r^{a_q}$. Let $a'_p, a'_q \in A$. Because $r\notin \{0,\pm 1\}$, $a_p+a_q =  a'_p+a'_q$ if and only if $r^{a_p+a_q} = r^{a'_p+a'_q}$. Therefore, $|B\cdot B| = |A+A|$. This completes our proof. 
\end{proof}

\begin{lem}\label{logtrans}
Let a MPTQ set $A$ of positive numbers be chosen. Fix $1\neq r>0$. Then 
$B  = log_r A $ is MSTD. 
\end{lem}
\begin{proof}
We will prove that $|B+B| = |A\cdot A|$ and $|B-B| = |A/A|$. Given a product $a_ia_j$ for some $a_i, a_j\in A$, we have the corresponding sum $\log_r a_i+\log_r a_j$ in $B+B$. Let $a'_i, a'_j$ be chosen. We have $a_ia_j = a'_ia'_j$ if and only if $\log_r a_i+\log_ra_j = \log_r a'_i + \log_r a'_j$. Hence, $|B+B| = |A\cdot A|$. Similarly, given a quotient $a_p/a_q$ for some $a_p, a_q\in A$, we have the corresponding difference $\log_r a_p - \log_r a_q$ in $B-B$. Let $a'_p, a'_q\in A$, We have $a_p/a_q = a'_p/a'_q$ if and only if $\log_r a_p - \log_r a_q = \log_r a'_p - \log_r a'_q$. Hence, $|B-B| = |A/A|$. This completes our proof. 
\end{proof}
\subsection*{Application: construction of an infinite family of MPTQ sets}
We can generate an infinite family of MSTD sets from a given MSTD set through the base expansion method. Let $A$ be a MSTD set, and let $$A_{k,m}=\bigg\{\sum_{i=1}^{k}a_im^{i-1}:a_i\in A\bigg\}.$$ If $m$ is sufficiently large, then $|A_{k,m}\pm A_{k,m}| = |A\pm
A|^k$ and $|A_{k,m}|=|A|^k$. The method is a very powerful tool and has been used extensively in the literature including \cite{He, ILMZ1, ILMZ2}. However, the base expansion method turns out to be inefficient in terms of our MSTD sets' cardinality. Due to Lemma \ref{expo} and Lemma \ref{logtrans}, we can use the base expansion method to generate an infinite family of MPTQ sets from a given MPTQ sets. 

Let $A$ be a MPTQ set of positive real numbers. By Lemma \ref{logtrans}, we know that $\log_2 A$ is a MSTD set. Now, apply the base expansion method to generate an infinite family of MSTD sets from $\log_2 A$. Due to Lemma \ref{expo}, if $B$ is a MSTD set in the family, we know that $2^B$ is a MPTQ set. 
%%%%%%%%%%%%%%%%%%%%%%%%%%%%%%%%%%%%%%%%%%%%%%%%%%%%%%%%%%%%%%%%%%%%%%%%%%%%%%%%%%%%%%%%%%%%%%%%%%%%%%%%%%
%%%%%%%%%%%%%%%%%%%%%%%%%%%%%%%%%%%%%%%%%%%%%%%%%%%%%%%%%%%%%%%%%%%%%%%%%%%%%%%%%%%%%%%%%%%%%%%%%%%%%%%%%%
%%%%%%%%%%%%%%%%%%%%%%%%%%%%%%%%%%%%%%%%%%%%%%%%%%%%%%%%%%%%%%%%%%%%%%%%%%%%%%%%%%%%%%%%%%%%%%%%%%%%%%%%%%
\section{The smallest MPTQ set}
%%%%%%%%%%%%%%%%%%%%%%%%%%%%%%%%%%%%%%%%%%%%%%%%%%%%%%%%%%%%%%%%%%%%%%%%%%%%%%%%%%%%%%%%%%%%%%%%%%%%%%%%%%
\begin{proof}[Proof of Theorem \ref{smallestcar} item 1]
We prove by contradiction. Let $A$ be a MPTQ set with $|A|\le 7$. By Lemma \ref{logtrans}, $B = \log_2 A$ is MSTD and $|B| = 7$. This contradicts \cite[Theorem 6]{Na1}. So, $|A|\ge 8$, as desired. 
\end{proof}
\begin{exa}\normalfont
An example of a MPTQ set with cardinality $8$ is 
$$S_4 \ = \ \{2^0, 2^2, 2^3, 2^4, 2^7, 2^{11}, 2^{12}, 2^{14}\}.$$
This set is the $2$-exponential transformation of the MSTD set $\{0,2,3,4,7,11,12,14\}$. Lemma \ref{expo} guarantees that $S_4$ is MPTQ. 
\end{exa}
The restriction we have in Theorem \ref{smallestcar} item 1 is that our MPTQ set only contain positive numbers. Next, we relax this condition to prove Theorem \ref{smallestcar} item 2.
We employed the same technique used by the author \cite{CI2} with a nontrivial modification of the proof for the product/quotient case. The proof is more complicated compared to the proof of \cite[Theorem 1]{CI2} because of interactions between negative and positive numbers. The next lemma follows from \cite[Proposition 7]{CI2} and the proof of Lemma \ref{expo}.
\begin{lem}\label{notfar}
Let $n\in \mathbb{N}$ and $r\in \mathbb{R}\backslash \{0,\pm 1\}$. Set $a = r^{(n-1)+k}$ for some $1\le k\le n-1$. Then $a\rightarrow G_{n,r}$ gives $k+1$ new products and $2k$ new quotients. 
\end{lem}
\begin{thm}\label{1joingeo}
Let $n\in\mathbb{N}$ and $r\in\mathbb{R}\backslash \{0,\pm 1\}$. For all $a\in \mathbb{R}\backslash \{0\}$, the set $G_{n,r}\cup \{a\}$ is not MPTQ.
\end{thm}

\begin{proof}
If $a\in G_{n,r}$, then we are done since $G_{n,r}$ is symmetric with respect to $r^{n-1}$ and thus, not MPTQ. For $n = 1$, we have $G_{1,r} =\{1,a\}$, which is symmetric with respect to $a$ and thus, not MPTQ. We assume that $a\notin G_{n,r}$ and $n\ge 2$. The number of new products as a result of $a\rightarrow G_{n,r}$ is at most $n+1$. We consider the following two cases. 

\noindent \textbf{Case 1:} $a = r^\ell$ for some $\ell\in \mathbb{N}_{>n-1}$. If $\ell = n$, we have $G_{n,r}\cup \{a\} = G_{n+1,r}$, which is not MPTQ. Consider $\ell\ge n+1$. Write $\ell = (n-1)+k$ for some $k\ge 2$.
\begin{itemize}
    \item If $2\le k\le n-1$, by Lemma \ref{notfar}, we have $k+1$ new products while $2k$ new quotients. So, our new set is not MPTQ.
    \item If $k> n-1$, then we have $2n$ new quotients. Since we have at most $n+1$ new products, our new set is not MPTQ. 
\end{itemize} 

\noindent \textbf{Case 2:} $a = r^{\ell}$ for some $\ell\in \mathbb{N}_{<0}$. Due to symmetry, this is similar to Case 1. 

\noindent \textbf{Case 3:} $a \neq r^\ell$ for all $\ell\in \mathbb{Z}$. Our set of new quotients contains
$$K \ =\ \bigg\{a,\frac{a}{r},\ldots,\frac{a}{r^{n-1}}\bigg\}.$$
\begin{itemize}
    \item If $1/a\in K$, then $a^2 \in G_{n,r}$. So, the number of new products is at most $n$. Because $|K| = n$, we know that our new set is not MPTQ.
    \item If $1/a\notin K$, then we have at least $n+1$ new quotients. Again, our new set is not MPTQ.
\end{itemize}
We have completed the proof. 
\end{proof}
\begin{cor}\label{geoprog}
A finite set containing numbers in a geometric progression in union with an arbitrary number is not MPTQ.
\end{cor}
\begin{proof}
Let our set be $A = \{a, ar, ar^2, \ldots, ar^{n-1},b\}$, where $n\in\mathbb{N}, ab\neq 0$, $r\notin \{0,\pm 1\}$. Then, $A/a = \{1, r, r^2, \ldots, r^{n-1}, b/a\} = G_{n,r}\cup \{b/a\}$, which is not MPTQ by Theorem \ref{1joingeo}. Hence, $A$ is not MPTQ.
\end{proof}

\begin{proof}[Proof of Theorem \ref{smallestcar} item 2]
Let $A$ be our finite set of positive numbers. We analyze $5$ cases corresponding to the cardinality of $A$.

\noindent Case 1: $|A| = 1$. Write $A = \{a_1\}$ for some $a_1\in \mathbb{R}\backslash \{0\}$. Because $A$ is symmetric with respect to $a_1^2$, $A$ is not MPTQ.

\noindent Case 2: $|A| = 2$. Write $A = \{a_1,a_2\}$ for some $a_1,a_2\in\mathbb{R}\backslash \{0\}$. Because $A$ is symmetric with respect to $a_1a_2$, $A$ is not MPTQ.

\noindent Case 3: $|A| = 3$. Write $A = \{a_1,a_2,a_3\}$ for some $a_1,a_2,a_3\in \mathbb{R}\backslash \{0\}$. Consider $A/a_1 = \{1,a_2/a_1,a_3/a_1\}$. Either $a_2/a_1\neq -1$ or  $a_3/a_1\neq -1$. Without loss of generality, assume that $a_2/a_1\neq -1$. Because $\{1, a_2/a_1\} = G_{2, a_2/a_1}$, Theorem \ref{1joingeo} says that $A/a_1 = G_{2,a_2/a_1} \cup \{a_3/a_1\}$ is not MPTQ. Hence, $A$ is not MPTQ. 

\noindent Case 4: $|A| = 4$. Write $A = \{a_1, a_2, a_3, a_4\}$ for some $0<|a_1|\le |a_2|\le |a_3|\le |a_4|$. By Proposition \ref{trivialbounds}, we know that $\max |A\cdot A| = 10$, while $\max |A/A| = 13$.  Since we have only $4$ numbers, we do not have $a_i \cdot  a_j = a_m \cdot a_n = a_p\cdot a_q$ for all $a_i, a_j, a_m, a_n, a_p, a_q\in A$ and $a_i,a_j, a_m, a_p, a_q$ being pairwise different. Let $$k=\sum_{q\in A/A, q\neq 1, |q|\ge 1}(|(A/A)_q| - 1),$$
then we can apply Remark \ref{deepex} Case 1 to have $$\sum_{p\in A\cdot A}(|(A\cdot A)_p|-1)\ \ge\ k/2.$$ In order that $A$ is MPTQ, it must be that
\begin{align}
    13 - 2k \ <\ 10 - k/2 \label{case4}.
\end{align}
Solving for $k$, we have $k\ge 3$. Therefore, $|A/A| \le 13 - 6 = 7$. For $1\le i\le 3$, set $m_i = a_{i+1}/a_i$. Note that $|m_i| \ge 1$ and $m_i \neq 1$. Then
$$A = (a_1\,|\,m_1,m_2,m_3).$$
We have $6$ distinct quotients 
$$K \ =\ \{1, m_1, m_1m_2, m_1m_2m_3, (m_1m_2)^{-1}, (m_1m_2m_3)^{-1}\}.$$

Subcase 4.1: $m_1\neq -1$. Then $(m_1)^{-1}$ is another distinct quotient. Because $|A-A|\le 7$, we have $m_2\in K\cup\{(m_1)^{-1}\}$. The only possible option is that $m_2 = m_1$. Then $\{a_1,a_2,a_3\}$ is a geometric progression. By Corollary \ref{geoprog}, $A$ is not MPTQ.

Subcase 4.2: $m_1 = -1$. Then $m_2\neq m_1$ because if not, $m_1m_2 = 1$ or $a_1 = a_3$, a contradiction. Either $m_2\notin K$ or we have $m_2 \in \{m_1m_2m_3, (m_1m_2m_3)^{-1}\}$.
\begin{itemize}
    \item Subcase 4.2.1: $m_2\notin K$. Then $(m_2)^{-1}\in K\cup\{m_2\}$. The only option is $(m_2)^{-1} \in\{ (m_1m_2m_3)^{-1}, m_1m_2m_3\}$. So, $m_3 = -1$. Our set  $$A\ =\ \{a_1,-a_1,-a_1m_2,a_1m_2\},$$ which is symmetric with respect to $a_1^2m_2$ and thus, not MPTQ.
    \item Subcase 4.2.2: $m_2\in K$. The only option is $m_2 \in \{(m_1m_2m_3)^{-1},m_1m_2m_3\}$, or equivalently, $m_1m_3 = 1$. Again, we have $m_3 = -1$. According to Subcase 4.2.1, our set is not MPTQ. 
\end{itemize}

We complete our proof that $|A|\ge 5$. 
\end{proof}
%%%%%%%%%%%%%%%%%%%%%%%%%%%%%%%%%%%%%%%%%%%%%%%%%%%%%%%%%%%%%%%%%%%%%%%%%%%%%%%%%%%%%%%%%%%%%%%%%%%%%%%%%%
%%%%%%%%%%%%%%%%%%%%%%%%%%%%%%%%%%%%%%%%%%%%%%%%%%%%%%%%%%%%%%%%%%%%%%%%%%%%%%%%%%%%%%%%%%%%%%%%%%%%%%%%%%
%%%%%%%%%%%%%%%%%%%%%%%%%%%%%%%%%%%%%%%%%%%%%%%%%%%%%%%%%%%%%%%%%%%%%%%%%%%%%%%%%%%%%%%%%%%%%%%%%%%%%%%%%%
\section{Sequences with no MPTQ subsets}

\begin{proof}[Proof of Theorem \ref{thm:gen}]
Let $S = \{s_1, s_2, . . . , s_k\} = \{a_{g(1)}, a_{g(2)}, . . . , a_{g(k)}\}$ be a finite subset of $A$, where $g: \mathbb{Z}^+\rightarrow\mathbb{Z}^+$ is a strictly increasing function. We show that $S$ is not MPTQ
by strong induction on $g(k)$. 

For the base case, we know that all MPTQ sets have at least $5$ elements due to Theorem \ref{smallestcar} item 2,
so any subset $S$ of $A$ with exactly $k$ elements is not a MPTQ set if $k \le 4$; in particular,
$S$ is not a MPTQ set if $g(k)\le 4$. Thus we may assume for $g(k) \ge 5$ that all $S'$ of the
form $\{s_1, . . . , s_{k-1}\}$ with $|s_{k-1}| \le |a_{g(k)}|$ are not MPTQ sets. The proof is completed by
showing $S = S' \cup \{a_{g(k)}\} = \{s_1, . . . , s_{k-1}, a_{g(k)}\}$ is not MPTQ sets for any $a_{g(k)}$.

For the inductive step,  $S'$ is not a MPTQ set by the inductive assumption. If $k \le 2r-1$ then $|S| \le 2r-1$ and $S$ is not a MPTQ set by the second assumption of the theorem. 
If $k\ge 2r$, consider the number of new products and quotients obtained by adding $a_{g(k)}$. As we have at most 
$k$ new products, we are done if there are at least $k$ new quotients.

Since $k \ge 2r$, we have $k - \floor{\frac{k+1}{2}} \ge r$. Let $t = \floor{\frac{k+1}{2}}$. Then $t \le k - r$, which implies $|s_{t}| \le |s_{k-r}|$. The largest quotient in absolute value between elements in $S'$ is $|s_{k-1}/s_1|$ and the smallest is $|s_1/s_{k-1}|$; we now show that we have added at least $k$ distinct quotients whose absolute values are either greater than $|s_{k-1}/s_1|$ or smaller than $|s_1/s_{k-1}|$, which will complete the proof. We have
\begin{align*}
|a_{g(k)}/s_{t}| &\ \ge \ |a_{g(k)}/s_{k-r}| \ = \ |a_{g(k)}/a_{g(k-r)}|  \nonumber\\
                   &\ \ge \ |a_{g(k)}/a_{g(k)-r}| \nonumber\\
                   &\  >  \ |a_{g(k)-1}/a_1| & \text{(by the first assumption on $\{a_n\}$)} \nonumber\\
                   &\ \ge \ |s_{k-1}/a_{1}| \ = \  |s_{k-1}/s_{1}|.
\end{align*}
Since $|a_{g(k)}/s_t| > |s_{k-1}/s_1|$, we know that
$$a_{g(k)}/s_t, \ldots, a_{g(k)}/s_2, a_{g(k)}/s_1$$
are $t$ quotients whose absolute values are greater than $|s_{k-1}/s_1|$. As we could do division in the opposite order, we have $t$ quotients who absolute values are smaller than $|s_1/s_{k-1}|$. Therefore, the total number of new quotients is at least
$$2t \ =\ 2\bigg\lfloor\frac{k+1}{2}\bigg\rfloor\ \ge \ k.$$
This completes our proof. 
\end{proof}

\begin{proof}[Proof of Theorem \ref{logprimenoMSTD}]
We first prove item 1. Consider $A = \{a_1, a_2,\ldots, a_n\}\subset P$ for some $n\in \mathbb{N}$ and $a_1 < a_2< \cdots < a_n$. Due to Theorem \ref{smallestcar} item 1, it suffices to prove the following claim: if $A\backslash\{a_n\}$ is not MPTQ, then $A$ is not MPTQ. In particular, we will prove that $a_{n}\rightarrow A\backslash\{a_n\}$ gives more new quotients than new products. Clearly, $a_{n}\rightarrow A\backslash\{a_n\}$ gives at most $n$ new products. The following are new quotients
$$\frac{a_{n}}{a_1},\frac{a_{n}}{a_2},\ldots, \frac{a_{n}}{a_{n-1}}.$$
Indeed, suppose that $a_{n}/a_j = a_m/a_k$ for some $1\le m,k,j\le n-1$. Then $a_{n}a_k = a_ma_j$, implying that either $a_m|a_k$ or $a_m|a_{n}$, which contradicts that $a_k,a_{n}\in P$. Hence, we have $n-1$ new quotients greater than $1$. Their reciprocals must also be new. Therefore, we have $2(n-1)$ new quotients. For $n\ge 8$, $2(n-1)>n$, and so, $A$ is not MPTQ. Again, the reason we only concern with $n\ge 8$ is due to Theorem \ref{smallestcar} item 1.

We proceed to prove item 2. Fix $r>0$ and $r\neq 1$. We prove by contradiction. Suppose that $P_r$ contains a MSTD subset $A$. By Lemma \ref{expo}, $r^{A}\subset P$ is MPTQ, implying that $P$ contains a MPTQ subset. This contradicts item 1 above. 
\end{proof}
\section{Questions}
We end with a list of questions for future research.
\begin{itemize}
    \item The diameter of a set is defined to be the difference between the maximum and the minimum. What is the smallest diameter of a MPTQ sets? What is the smallest $n$ such that $\{1,2,\ldots,n\}$ has a MPTQ subset? 
    \item Can we construct MPTQ sets explicitly without using MSTD sets and Lemma \ref{expo}? A conventional method of constructing MSTD sets is to fix a fringe pair $(L,R)$ of two sets containing elements to be used in the fringe of the interval and argue that all the middle elements will appear. The fringe pair ensures that some of the largest and smallest differences are missed and that our set is MSTD. For MSTD sets, you can use the fringes since it is possible to manage the interaction (addition and subtraction) of numbers in the middle. For example, $(n-9) + 4 = (n-10) + 5$. However, for multiplication, $4(n-9)$ is not necessarily equal to $5(n-10)$. Because it is hard to work with the middle of MPTQ sets, it is not clear how to use the fringes to construct MPTQ sets as a result. 
    \item Is there a set that is both MSTD and MPTQ? 
\end{itemize}
\subsection*{Acknowledgement}
The author would like to thank Carlo Sanna for providing the proof of Theorem \ref{goto0}. Many thanks to the anonymous referee for useful comments that help improve this paper.

\bigskip

\end{document}